\documentclass[sts]{imsart-mod}

\usepackage{amsthm,amsmath, amssymb, color}

\usepackage[utf8]{inputenc}
\usepackage{amsmath}
\usepackage{amssymb}
\usepackage{bbm, bm}
\usepackage{enumitem}

\usepackage{hyperref}
\hypersetup{
  colorlinks = true,
  urlcolor = blue, 
  linkcolor = blue,
  citecolor = blue}

\usepackage{natbib}

\usepackage{color}
\usepackage{amsfonts, fancybox}
\usepackage{amssymb}
\usepackage[american]{babel}
\usepackage{graphicx, pstricks, pst-node, color}
\usepackage{psfrag,color}
\usepackage{dcolumn}
\usepackage{subfigure}
\usepackage{flafter, bm, dsfont}
\usepackage[section]{placeins}

\usepackage{multirow}
\usepackage{color}
\usepackage{amsmath}
\usepackage{float}
\usepackage{mathrsfs}
\usepackage{amsfonts}
\usepackage{dsfont}
\usepackage{amssymb}
\usepackage[boxed]{algorithm2e}
\usepackage{amssymb, latexsym}

\usepackage{graphicx}
\usepackage{epstopdf}

\usepackage{boxedminipage}

\newcommand{\bc}{\begin{center}}
\newcommand{\ec}{\end{center}}
\newcommand{\ba}{\begin{array}}
\newcommand{\ea}{\end{array}}
\newcommand{\be}{\begin{eqnarray}}
\newcommand{\ee}{\end{eqnarray}}
\newcommand{\bel}{\begin{eqnarray}\label}
\newcommand{\El}{\end{eqnarray}}
\newcommand{\bes}{\begin{eqnarray*}}
\newcommand{\Es}{\end{eqnarray*}}
\newcommand{\bn}{\begin{enumerate}}
\newcommand{\en}{\end{enumerate}}

\definecolor{MIT}{cmyk}{.24, 1.00, .78, .17} 
\definecolor{pink}{cmyk}{0, 1, 0, 0} 
\definecolor{darkgreen}{cmyk}{1,0, 1, 0} 

\newtheorem{lemma}{Lemma}
\newtheorem{definition}{Definition}

\newtheorem{theorem}{Theorem}

 \newtheorem*{theorem*}{Theorem}
 \newtheorem*{lemma*}{Lemma}
 \newtheorem*{proposition*}{Proposition}
 \newtheorem*{corollary*}{Corollary}

{
	\begin{center}
		\begin{boxedminipage}{0.8\linewidth}
			\begin{center}
				\textbf{\texttt{#1}}
			\end{center}
			\rm
			\begin{tabbing}
				....\=...\=...\=...\=...\=  \+ \kill
			} %
			{\end{tabbing}
	\end{boxedminipage} \end{center} 
}

\newcommand{\flo}[1]{\lfloor #1 \rfloor} 
\newcommand{\eps}{\varepsilon} 
\newcommand{\mb}[1]{\mathbb{#1}} 
\newcommand{\mc}[1]{\mathcal{#1}} 
\newcommand{\abs}[1]{\left|#1\right|} 
\newcommand{\norm}[1]{\lVert#1\rVert} 
\newcommand{\V}[1]{^{\vec{#1}} } 
\newcommand{\J}{_{\vec{j}} } 
\newcommand{\hd}{ \mc{P}_\mc{H} } 
\newcommand{\ti}{\tilde } 

\newcommand{\cbeta}{\lfloor \beta \rfloor}

\newcommand{\cF}{\mathcal{F}}

\newcommand{\cH}{\mathcal{H}}

\newcommand{\cP}{\mathcal{P}}

\newcommand{\cV}{\mathcal{V}}

\newtheorem{assumption}{Assumption}

\newcommand{\h}{{\rm I}\kern-0.18em{\rm H}}
\newcommand{\K}{{\rm I}\kern-0.18em{\rm K}}
\newcommand{\p}{\mb{P}}
\newcommand{\E}{\mb{E}}
\newcommand{\Z}{{\rm Z}\kern-0.18em{\rm Z}}
\newcommand{\1}{{\rm 1}\kern-0.24em{\rm I}}
\newcommand{\N}{{\rm I}\kern-0.18em{\rm N}}

\begin{document}
	
	\begin{frontmatter}
		\title{Efficient Interpolation of Density Estimators}
		\runtitle{Interpolation}
		
				    
		\author{ 
			\fnms{Paxton}
			\snm{Turner}\ead[label=paxton]{pax@mit.edu},
			\fnms{Jingbo}
			\snm{Liu}\ead[label=jingbo]{jingbol@illinois.edu},
			and
			\fnms{Philippe} \snm{Rigollet}
			\thanks{P.R. was supported by NSF awards IIS-1838071, DMS-1712596, DMS-1740751, and DMS-2022448.}
			\ead[label=rigollet]{rigollet@math.mit.edu}
		    }
		
		\affiliation{Massachusetts Institute of Technology and University of Illinois at Urbana-Champaign }
		
		
		%
		\address{{Paxton Turner}\\
			{Department of Mathematics} \\
			{Massachusetts Institute of Technology}\\
			{77 Massachusetts Avenue,}\\
			{Cambridge, MA 02139-4307, USA}\\
			\printead{paxton}
		}

		\address{{Jingbo Liu}\\
		    {Department of Statistics} \\
			{University of Illinois at Urbana-Champaign}\\
			{725 S. Wright St.,}\\
			{Champaign, IL 61820, USA} \\
			\printead{jingbo}
		}
		
		\address{{Philippe Rigollet}\\
			{Department of Mathematics} \\
			{Massachusetts Institute of Technology}\\
			{77 Massachusetts Avenue,}\\
			{Cambridge, MA 02139-4307, USA}\\
			\printead{rigollet}
		}

		
		\runauthor{Turner et al.}
		
		\begin{abstract}
	    We study the problem of space and time efficient evaluation of a nonparametric estimator that approximates an unknown density. In the regime where consistent estimation is possible, we use a piecewise multivariate polynomial interpolation scheme to give a computationally efficient construction that converts the original estimator to a new estimator that can be queried efficiently and has low space requirements, all without adversely deteriorating the original approximation quality. Our result gives a new statistical perspective on the problem of fast evaluation of kernel density estimators in the presence of underlying smoothness. As a corollary, we give a succinct derivation of a classical result of Kolmogorov---Tikhomirov on the metric entropy of H\"{o}lder classes of smooth functions.
		\end{abstract}
		
		
	\end{frontmatter}

\section{\uppercase{Introduction}}
	The fast evaluation of kernel density estimators has been well-studied including approaches based on the fast Gauss transform \cite{GreStr91}, hierarchical space decompositions \cite{GreRok87}, locality sensitive hashing \cite{ChaSim17,BacChaInd18,SimRonBai19,BacIndWag19}, and binning \cite{Sco85}, as well as interpolation \cite{Jon98,Kog98}, our main technique in this work. Typically these techniques carefully leverage the structure of the kernel under consideration, and many of them operate in a worst-case framework over the dataset. In this work, we consider the problem of fast evaluation of a density estimator $\hat f$ in a statistical setting where $\hat f$ gives a good pointwise approximation to an unknown density $f$  that lies in a H\"{o}lder class of smooth functions. We show that a pointwise approximation guarantee alone, without assuming any specific structure of the estimator $\hat f$, is enough to construct a new estimator $\tilde f$ that can be stored and queried cheaply, and whose approximation error is similar to that of the original estimator. Our approach is based on a multivariate polynomial interpolation scheme of \citet{Nic72} \citep[see also][]{ChuYao77} and provides an explicit formula for $\tilde f$ in terms of some judiciously chosen queries of the original estimator.
    
	\subsection{Background and related work}
	
    Density estimation is the task of estimating an unknown density $f$ given an i.i.d. sample $X_1, \ldots, X_n \sim \p_f$, where $\p_f$ is the probability distribution associated to $f$. A popular choice of density estimator is the kernel density estimator (KDE)
    
    \begin{equation}
        \label{EQ:KDE}
        \hat f(y):=\frac{1}{nh^d}\sum_{j=1}^n K\left(\frac{X_i-y}{h}\right).
    \end{equation}
    
    With proper setting of the bandwidth parameter $h$ and choice of kernel $K$, the KDE $\hat f$ is a minimax optimal estimator over the $L$-H\"{o}lder smooth densities $\mc{P}_{\mc{H}}(\beta, L)$ of order $\beta$ \citep[see e.g.][Theorem~1.2]{Tsy09}:
    
    \begin{equation}
            \label{eqn:classical_risk}
             \inf_{ \hat{f} } \sup_{ f \in \mc{P}_{\mc{H}}(\beta, L)  }\mb{E}_f \, \norm{ \hat{f} - f }_2    = \Theta_{\beta, d, L}( n^{-\frac{\beta}{2 \beta + d} } )\,.
    \end{equation}
     
    Despite its statistical utility, the KDE \eqref{EQ:KDE} has the computational drawback that it naively requires $\Omega(n)$ time to evaluate a query. The problem of improving on these computational aspects has thus received a lot of attention.
    
    Motivated by multi-body problems, \cite{GreStr91} developed the fast Gauss transform to rapidly evaluate sums of the form \eqref{EQ:KDE} when $K(x) = \exp(-\abs{x}_2^2)$ is the Gaussian kernel. Their work is posed a worst-case batch setting where $\hat f$ is to be evaluated at $m$ points $y_1, \ldots, y_m$ specified in advance and the locations $X_1, \ldots, X_n$ lie in a box. Their techniques use hierarchical space decompositions and series expansions to show that \eqref{EQ:KDE} may be evaluated at $y_1, \ldots, y_m$ with precision $\eps$ in time $h^d (\log \frac{1}{\eps})^d (n + m)$. These results apply to any kernel that has a rapidly converging Hermite expansion \citep[see also][]{GreRok87}. There are also follow up works on the improved fast Gauss transform and tree-based methods that use related ideas \citep{YanDurGum03,LeeMooGra06}.  
   
    More recently, several works \citep{ChaSim17,BacChaInd18,SimRonBai19,BacIndWag19} are devoted to the problem of fast evaluation of \eqref{EQ:KDE} in high dimension using locality sensitive hashing. In these works, the dataset is carefully reweighted for importance sampling such that a randomly drawn datapoint $X_r$'s corresponding kernel value $K(X_r - y)$ gives a good approximation to $\hat f (y)$. This sampling procedure can be executed efficiently using hashing-based methods. For example, \citet{BacIndWag19} show that for the Laplace and Exponential kernels with bandwidth $h = 1$, e.g., the value $\hat f(y)$ can be computed with multiplicative $1 \pm \eps$ error in time $O(\frac{d}{\sqrt{\tau}\eps^2})$ even in worst case over the dataset, where $\tau$ is a uniform lower bound on the KDE.
   
    Another effective approach to this problem in high dimensions is through coresets \citep{AgaHarVar05,Cla10,PhiTai18,PhiTai19}. A coreset is a representative subset $\{X_i\}_{i \in S}$ of a dataset such that
    \[
    \hat f(y) \approx \frac{1}{nh^d} \sum_{i \in S} K\left(\frac{X_i-y}{h}\right). 
    \]
    When $h = O(1)$, for example, the results of \citet{PhiTai19} give a polynomial time algorithm in $n, d$ such that the coreset KDE yields an additive $\eps$ approximation to $\hat f$ using a coreset of size $\tilde{O}(\frac{\sqrt{d}}{\eps})$. Their results hold in worst case over the dataset and for a variety of popular kernels. The methods of \citet{PhiTai19} are powered by state-of-the-art algorithms from discrepancy theory \citep{BanDadGarLov18} \cite[see][for a comprehensive exposition on discrepancy]{Mat99,Cha00}.
    
    Our approach is most closely related to prior work on the interpolation of kernel density estimators due to \citet{Jon98} and \citet{Kog98}. Motivated by visualization and computational aspects, \citet{Jon98} studies binned and piecewise linearly interpolated univariate kernel density estimators and provides precise bounds on the mean-integrated squared error. \citet{Kog98} extends this work and constructs higher order piecewise polynomial interpolants of multivariate kernel density estimators, and shows that for very smooth densities, this procedure improves the mean-integrated squared error. We also note the recent work of \cite{BelRakTsy19,LiaRak20} demonstrating the perhaps surprising effectiveness of interpolation in nonparametric regression. 
    
    Our work differs from \cite{Kog98} in a few important respects. We do not assume $\hat f$ to be a KDE in the first place, but rather give a general method for effectively interpolating a minimax density estimator. Also, our results hold for the entire range of smoothness parameters $\beta$ and dimension $d$, while \cite{Kog98} requires the density to be at least $qd$ times differentiable when interpolating KDEs with kernels of order $q$ \citep[Definition 1.3]{Tsy09}. On the other hand, our method increases the mean squared by a multiplicative factor $\tilde{O}(c_{\beta, d}),$ while Kogure's approach improves the mean squared error (though our focus here is the sup norm). Finally, we use a different interpolation scheme as detailed in Section \ref{sec:principal_lattice}.

    \subsection{Results} 
    We seek to impose minimal requirements on a density estimator $\hat f$ of an unknown density $f$ so that it can be converted to a new estimator $\tilde{f}$ that performs well on the following criteria.
        
    \begin{enumerate}[itemsep=3pt,parsep=3pt]
    \item (\textbf{Minimax}) $\tilde{f}$ is a minimax estimator for $f$
    \item  (\textbf{Space-efficient}) $\tilde{f}$ can be stored efficiently
    \item (\textbf{Fast querying}) $\tilde{f}$ can be evaluated efficiently
    \item (\textbf{Fast preprocessing}) $\tilde{f}$ can be constructed efficiently
    \end{enumerate}
    In the statistical setup where typically $\beta, d = O(1)$, by \textit{efficient} we mean requiring only polynomial time or space in the sample size $n$. In particular for fixed $\beta$, by \eqref{eqn:classical_risk} consistent estimation is only possible when $d \ll \log n$. In what follows we indicate dependencies on the parameters $\beta$ and $d$.
    
    The requirement that we place on the estimator $\hat f$ to be converted is the following assumption. 
    
    \begin{assumption}
         \label{ass:tail_bound}
         For all $y \in [0, 1]^d$ and $1 \geq t \geq  \eps$, we have 
             \begin{equation*}
             	\label{eqn:pointwise_tail_bound}
             	\sup_{f \in \hd(\beta, L)} \mb{P}_f\left[ \abs{ \hat f(y) - f(y)} > t \right] \leq 
             	2 \exp\left( - \frac{t^2}{\eps^2} \right),
             \end{equation*}
         where $\eps := c^* \, n^{-\beta/(2\beta + d)}$ is the minimax rate over all density estimators and $c^* = c_{\beta, d, L} > 0$. 
    \end{assumption}
     
    In particular this is satisfied if the pointwise error is a sub-Gaussian random variable with parameter $\eps$ that captures the minimax rate of estimation. For the KDE built from a kernel $K$ of order $\ell:= \flo{\beta}$ \citep[][Definition 1.3]{Tsy09} and bandwidth $h = n^{-1/(2\beta + d)}$, \eqref{eqn:pointwise_tail_bound} follows from a standard bias-variance trade-off and an application of Bernstein's inequality for bounded random variables \citep[see e.g.][]{Ver18}. Under Assumption \ref{ass:tail_bound} we have our main result.  
    
\begin{theorem}
\label{thm:main}   
    Let $\hat f$ denote a density estimator satisfying Assumption \ref{ass:tail_bound} for some $\beta > 0$ and $d \geq 1$. Let $Q$ denote the amount of time it takes to query $\hat f$. Set $\ell = \flo{\beta}$. Then there exists an estimator $\ti f$ that can be constructed in time $O(Q \binom{\ell +d}{d} n^{\frac{d}{2\beta +d}})$, that requires $O(d^2 \ell^2 \binom{\ell +d}{d} n^{\frac{d}{2\beta +d}} \log n)$ bits to store, that can be queried in time $O(d \ell \binom{\ell + d}{d}\log n)$, and that satisfies
    \begin{equation*}
    \E_f \norm{\ti f - f}_\infty < \ti c (\log n)^{1/2} n^{-\frac{\beta}{2\beta +d}} ,
    \end{equation*}
    where we may take 
    \[
    \tilde{c} = 10 \binom{\ell + d}{\ell}^{3} (2 \ell)^{2 \ell} \left\{ c^* + \frac{Ld^{\beta}}{\ell!} \right\} + 2Ld^{\frac12 (3\ell + 1)}\,.
    \]
\end{theorem}

    In particular, for $\beta, d = O(1)$, we can evaluate queries to $\ti f$ in nearly constant time, and the estimator $\ti f$ can be stored using sublinear space. Moreover, $\ti f$ can be preprocessed in subquadratic time, assuming that the evaluation time of the original estimator $\hat f$ is $O_d(n)$, which holds for the KDE \eqref{EQ:KDE}. Finally, $\ti f$ is a minimax estimator in the sup norm up to logarithmic factors and thus minimax up to logarithmic factors in the $L_p$ norms. In addition, our construction of $\ti f$ in Section \ref{sec:principal_lattice} yields an explicit formula for $\ti f$ in terms of a sublinear number of queries of $\hat f$ at a set of judiciously chosen query points. Specifically, the estimator $\ti f$ is a piecewise multivariate interpolation of the estimator $\hat f$ on this lattice of query points. 
    	
	\subsection{Setup and notation}
	
    Fix an integer $d \ge 1$. For any  multi-index $s = (s_1,\ldots, s_d) \in \mathbb{Z}_{\geq 0}^d$, let $\abs{s} = s_1 + \cdots + s_d$ and for $x =(x_1, \ldots, x_d) \in \mathbb{R}^d$, define $s!=s_1!\cdots s_d!$ and $x^s=x_1^{s_1}\cdots x_d^{s_d}$. Let $D^s$ denote the differential operator 
    $$
    D^s=\frac{\partial^{|s|}}{\partial x_1^{s_1} \cdots \partial x_d^{s_d}}\,.
    $$
    Fix a positive real number $\beta,$ and let $\cbeta$ denote the maximal integer \textit{strictly} less than $\beta$. We reserve the notation $\norm{ \cdot }_p$ for the $L_p$ norm
    and $\abs{\cdot}_p$ for the $\ell_p$ norm. 
    
    Given $L>0$ we let  $\mc{H}(\beta, L)$ denote the space of H\"{o}lder functions $f:\mb{R}^d \to \mb{R}$ that are supported on the cube $[0,1]^d$, are $\cbeta$ times differentiable, and satisfy
    \begin{equation*}
    |D^s f(x) - D^s f(y)| \leq L \abs{x - y}_2^{\beta - \cbeta}\,,\quad 
    \end{equation*}
    for all $x, y \in \mb{R}^d$ and for all multi-indices $s$ such that $\abs{s} = \flo{\beta}$.
     
    Let $\mc{P}_{\mc{H}}(\beta, L)$ denote the set of probability density functions contained in $\mc{H}(\beta,L)$. For $f \in \mc{P}_{\mc{H}}(\beta,L)$, let $\mb{P}_f$ (resp. $\E_f$) denote the probability distribution (resp. expectation) associated to $f$. 
    
    The parameter $L$ will be fixed in what follows, so typically we write $\mc{P}_{\mc{H}}(\beta) := \mc{P}_{\mc{H}}(\beta,L)$. The constants $c, c_{\beta, d}, c_{L},$ etc. vary from line to line and their subscripts indicate parameter dependences.
    
\section{\uppercase{Efficient interpolation of density estimators}}
\label{sec:interpolation}

The important implication of Assumption \ref{ass:tail_bound} is that we can query $\hat f$ at a polynomial number of data points such that for each query $y$, $\hat f(y) \approx f(y)$, where $f$ is the unknown density. 

\begin{lemma}
\label{lem:union_bound} Let $A > 0$ and set $N = \Delta n^A$ with $\Delta >1$. Let $y_1, \ldots, y_N \subset [0,1]^d$ denote a set of points. Then with probability at least $1 - 2n^{-2}$,
\[
\abs{ \hat f(y_i) - f(y_i) } \leq \sqrt{ \log( 2 \Delta n^{A+2} ) } \, \, \eps  
\]
for all $1 \leq i \leq N$, where $\eps = c_{\beta, d, L} \, n^{-\beta/(2\beta + d)}$ is the minimax rate. 
\end{lemma}

\begin{proof}
Set $t = \sqrt{ \log( 2 \Delta n^{A+2} ) } \,  \eps \geq  \eps$ and apply Assumption \ref{ass:tail_bound} to $y_i$. Then by the union bound,
\[
\p\left[ \exists y_i\,:\, \abs{ \hat f(y_i) - f(y_i) } > t \right] \leq 2 \Delta n^A  e^{-\frac{t^2}{\eps^2}} \leq n^{-2}. 
\]
\end{proof}

We now describe our construction of $\ti f$. Define $\ell := {\cbeta}$ and $M = \binom{\ell + d}{\ell}$. 

\medskip

\noindent \textbf{Construction of} $\tilde{f}$ (\textit{informal}): 

\begin{enumerate}[itemsep=3pt,parsep=3pt]
	\item {\sc Partition}: Divide $[0,1]^d$ into $h^{-d}$ sub-cubes $\{I_{\vec{j}}\} \subset [0,1]^d$ of side-length $h = n^{-1/(2\beta + d)}$ where $\vec{j} \in \mb{Z}_{\geq 0}^d$ and $I_{\vec{j}} := [0, h]^d + h \vec{j}$. 
	\item {\sc Mesh}: For each $\vec{j}$, construct a mesh consisting of $M = \binom{\ell + d}{\ell}$ points $U^{\vec{j}}_1, \ldots, U^{\vec{j}}_M \in I_{\vec{j}}$. 
	\item {\sc Interpolate}: In each sub-cube $I_{\vec{j}}$, construct a multivariate polynomial interpolant $\hat q_{\vec{j}}$ on the $M$ points $( U_{1}\V{j} , \hat f(U^{\vec{j}}_1  )), \ldots, ( U_{M}\V{j} , \hat f( U^{\vec{j}}_M) )$. 
\end{enumerate}

\textbf{Return}: $\tilde{f}: [0,1]^d \to \mb{R}$ defined by 
\[
\tilde{f}(y) = \sum_{\vec{j}} \hat q_{\vec{j}}(y) \1(y \in I_{\vec{j}}). 
\]

We first give some intuition for why $\tilde{f}$ is an accurate estimator. On each sub-cube $I_{\vec{j}}$, the true density $f \in \mc{P}_{\mc{H}}(\beta,L)$ is approximated up to the minimax error by a polynomial $q_{\vec{j}}$ of degree at most $\ell$ by the properties of H\"{o}lder functions. Upon setting $\Delta = M$ and $A = d/(2\beta+d)$ in Lemma \ref{lem:union_bound}, this guarantees that for all queries $U\V{j}_k$ in the mesh, $\hat f(U\V{j}_k) \approx f(U\V{j}_k) \approx q_{\vec{j}}(U\V{j}_k)$ with high probability. By studying the stability of the resulting polynomial system of equations, we can show that this construction yields a good approximation to the `true' interpolation polynomial $q_{\vec{j}}$ on the sub-cube $I_{\vec{j}}$. This argument, carried out formally in the remainder of this section, yields the estimation bound of Theorem \ref{thm:main}.

Next, we comment on the remaining guarantees of Theorem \ref{thm:main}. As we show later, there is an explicit formula for $\hat q_{\vec{j}}$ so the main preprocessing bottleneck is evaluation $\hat f$ on the $M n^{d/(2\beta + d)}$ queries in the mesh, which naively takes $Q M n^{d/(2\beta + d)}$ time. For the space requirement, it suffices to store the queried values $\{ \hat f(  U\V{j}_k) \}$ up to polynomial precision as well as the elements of the mesh. Hence $O(d^2 \ell^2 \binom{\ell +d}{d} n^{\frac{d}{2\beta +d}} \log n)$ bits suffice, by the uniform boundedness of H\"{o}lder functions (see Lemma \ref{lem:bounded}) and Assumption \ref{ass:tail_bound}. Finally, to query $\hat f$ at a point $y \in [0,1]^d$ requires checking which sub-cube $y$ belongs to by scanning its $d$ coordinates and then evaluating $\hat q_{\vec{j}}(y)$, which is a $d$-variate polynomial of degree $\flo{\beta}$. These considerations lead to the guarantees of Theorem \ref{thm:main}. 


\subsection{Interpolation on the principal lattice}
  \label{sec:principal_lattice}
    
    To construct our interpolant, we refer to the following definition and theorem which are classical in finite element analysis \citep{Nic72,ChuYao77}. The lattice $\mc{P}_\ell$, dubbed the $\ell$-th principle lattice, has the special property that every function defined on $\mc{P}_\ell$ admits a unique polynomial interpolant of degree at most $\ell$. This translates to the invertibility of the associated Vandermonde-type matrix and, equivalently, means that $\mc{P}_\ell$ is not a subset of any algebraic hypersurface of degree at most $\ell$.
    	
    \begin{definition}[$\ell$-th principal lattice of $\Delta_d$]
    	\label{def:principal_lattice}
	Let $\Delta_d \subset [0,1]^d$ denote the simplex on the points $\{0\} \cup \{ e_i \}_{i = 1}^d \subset \mb{R}^d$, where $e_i$ denotes the $i$-th standard basis vector in $\mb{R}^d$. Label the vertices of $\Delta_d$ to be $v_0 = 0, v_i = e_i$ for $0 \leq i \leq d$. For all $x \in \mb{R}^d$, there exists a unique vector $(\lambda_0(x),\ldots, \lambda_d(x))$ with entries summing to one and such that
	\[
	x = \sum_{i = 0}^d \lambda_i(x) v_i\,.
	\]
	Let $\Lambda:\mb{R}^d \to \mb{R}^{d+1}$ denote the function such that $\Lambda(x)=(\lambda_0(x),\ldots,  \lambda_d(x))$.
	The \emph{$\ell$-th principal lattice} $\mc{P}_\ell$ of $\Delta_d$ is defined to be
	\begin{equation}
	\label{eqn:principal_lattice}
	\mc{P}_\ell = \left\{ x \in \Delta_d: \, \ell\Lambda(x) \in \mb{Z}^{d+1}_{\geq 0} \right\}\,.
	\end{equation} 
\end{definition}

The following classical result gives an explicit form for the interpolant.  
	
	\begin{theorem}[\cite{Nic72},\cite{ChuYao77}]
		\label{thm:classical_interpolate}
Write $\cP_\ell =\{U_1, \ldots, U_M\} \subset \Delta_d$ and let $g: \mc{P}_\ell \to \mb{R}$ denote a function defined on this lattice.
Define the polynomial
		\begin{equation}
		\label{eqn:basis_fn}
			p_i(x) = \prod_{\substack{t = 0 \\ \lambda_t(U_i) > 0 }}^{d} \prod_{r = 0}^{\ell  \lambda_t(U_i) - 1} \frac{ \lambda_t(x) - \frac{r}{\ell}}{ \lambda_t(U_i) - \frac{r}{\ell} }\,,
		\end{equation}
where we recall that $\lambda_t(x)$ is from Definition~\ref{def:principal_lattice}. Then 
		\[
			p(x) := \sum_{i = 1}^M p_i(x) g(U_i) 
		\]
		satisfies $p(U_i) = g(U_i)$ for all $U_i \in \mc{P}_\ell$. Moreover, this is the unique polynomial of degree at most $\ell$ with this property. 
	\end{theorem} 

	Since $\lambda_t(x)$ is linear in $x \in \mb{R}^d$, it is easy to see that $p_i(x)$ is a polynomial of degree $\ell$, and moreover $p_i(U_j)=1$ if $i=j$ and zero otherwise.  
	
	We are now ready to give a precise description of the construction of $\tilde{f}$. The idea is to generate the mesh for interpolation using a shifted and rescaled version of the $\ell$-th principal lattice on $\Delta_d \subset [0,1]^d$. Recall that $\hat f$ is a density estimator that satisfies \eqref{eqn:pointwise_tail_bound}.  
    
	\medskip
	
\noindent \textbf{Construction of} $\tilde{f}$ (\textit{formal version}): 

\begin{enumerate}[itemsep=3pt,parsep=3pt]
	\item \textsc{Partition}: Divide $[0,1]^d$ into $h^{-d}$ sub-cubes $\{I_{\vec{j}}\} \subset [0,1]^d$ of side-length $h = n^{-1/(2\beta + d)}$ where $\vec{j} \in \mb{Z}_{\geq 0}^d$ and $I_{\vec{j}} := [0, h]^d + h \vec{j}$. 
	\item \textsc{Mesh}: For each $\vec{j}$, construct a mesh on $I_{\vec{j}}$ consisting of $M = \binom{\ell + d}{\ell}$ points given by the shifted and rescaled principal lattice $\mc{P}_\ell^{\vec{j}} := \{ h(x + \vec{j}) :\, x \in \mc{P}_\ell  \} \subset I_{\vec{j}}$. Let $U^{\vec{j}}_1, \ldots, U^{\vec{j}}_M$ denote the points in $\mc{P}_\ell^{\vec{j}}$. 
    
	\item \textsc{Interpolate}: In each sub-cube $I_{\vec{j}}$, construct a multivariate polynomial interpolant $\hat{q}_{\vec{j}}$ through the $M$ points $( U_{1}\V{j} , \hat f(U_{1}\V{j}), \ldots, ( U_{M}\V{j} , \hat f(U_{M}\V{j}) )$ given by $\hat{q}_{\vec{j}}(y) = p(y/h - \vec{j})$, where $p$ is the polynomial interpolant from Theorem \ref{thm:classical_interpolate} given by
	\[
		p(x) = \sum_{k = 1}^M p_k(x) \hat f( U^{\vec{j}}_k ) .
	\]
	
\end{enumerate}

\textbf{Return}: $\tilde{f}: [0,1]^d \to \mb{R}$ defined by 
\[
\tilde{f}(y) = \sum_{\vec{j}} \hat{q}_{\vec{j}}(y) \1(y \in I_{\vec{j}}). 
\]

The interpolant constructed in Step (4) is unique by Theorem \ref{thm:classical_interpolate}. Also recall that in each sub-cube $I_{\vec{j}}$, the true density $f \in \mc{P}_{\mc{H}}(\beta)$ is well-approximated by a polynomial $q_{\vec{j}}$ of degree $\ell$. To prove  Theorem~\ref{thm:main} it remains to show that given the values of this polynomial (up to some small error incurred by the estimation error of $\hat f$) on the points in our mesh, the interpolant $\hat{q}_{\vec{j}}$ from Step (4) gives a good approximation to $q_{\vec{j}}$ on the sub-cube $I_{\vec{j}}$.

\subsection{Proof of Theorem~\ref{thm:main}}
\label{sec:proof_interpolation}
First, we quantify the error in the approximation of the values of $q_{\vec{j}}$ on the mesh points. Let $f_{z, \ell}$ denote the degree $\ell$ polynomial given by the Taylor expansion of $f \in \mc{P}_{\mc{H}}(\beta)$ at $z$. Since $f \in \mc{P}_{\mc{H}}(\beta)$, by a standard fact (see Lemma \ref{lem:taylor_holder}) it holds that
\[
|f(y) - f_{z, \ell}(y)| \leq \frac{L d^{\ell/2} }{\ell !} \abs{y-z}_2^\beta, 
\]
where $f_{z, \ell}$ is the degree-$\ell$ Taylor expansion of the function $f$ at $z \in \mb{R}^d$.

For $\vec{j} \in \{0, \ldots, h^{-1} - 1\}^d$, define $q_{\vec{j}} := f_{z_{\vec{j}}, \ell}$, where $z_{\vec{j}}$ is the vertex of $I_{\vec{j}}$ closest to the origin. Then for all $y \in I_{\vec{j}}$, it holds that
\begin{align}
\label{eqn:Taylor_approx}
|f(y) - q_{\vec{j}}(y)| &\leq \left( \frac{L d^{\beta} }{\ell !} \right) \, h^{\beta} \nonumber \\
&= \left(\frac{L d^{\beta} }{\ell !}\right) n^{-\beta/(2\beta + d)} \nonumber \\
&=: \hat c n^{-\beta/(2\beta + d)}
\end{align}
Note that the right-hand side is the minimax rate of estimation in \eqref{eqn:classical_risk} up to constant factors.

Next, by Lemma \ref{lem:union_bound} and \eqref{eqn:Taylor_approx} it holds with probability at least $1 -  n^{-2}$ that 
\begin{align}
\label{eqn:noisy_interpolate}
\abs{ q_{\vec{j}}(U^{\vec{j}}_k) - \hat f(U^{\vec{j}}_k) }
&\leq (c^*\log 3M + \hat c) (\log n)^{\frac{1}{2}} n^{-\frac{\beta}{2\beta + d}} \nonumber \\ 
&=: \breve{c} (\log n)^{\frac{1}{2}} n^{-\frac{\beta}{2\beta + d}} 
\end{align}
for all $\vec{j} \in \{0, \ldots, h^{-1} -1\}^d$ and $ k \in [M]$. Using this fact, we can show that the polynomial interpolant built on $\{ (U_k^{\vec{j}}, \hat f( U_k^{\vec{j}} ))\}_{k = 1}^{M}$ provides a good approximation for $q_{\vec{j}}$ on the interval $I_{\vec{j}}$, which is our next task. The following lemma establishes stability of the polynomial approximation.

\begin{lemma}
	\label{lem:interpolation_error}
	Let $\hat{q}_{\vec{j}}$ denote the unique polynomial of degree at most $\ell$ that passes through the points $\{( U_k^{\vec{j}}, \hat f( U_k^{\vec{j}} )) \}_{k = 1}^{M}$. Then with probability at least $1 - 2 n^{-2}$, for all $\vec{j}$ and all $x \in I_{\vec{j}}$,
	\begin{equation}
	\label{eqn:interpolation_error}
	\abs{ q_{\vec{j}}(x) - \hat{q}_{\vec{j}}(x) } \leq c_{\beta, d, L} (\log n)^{\frac{1}{2}} n^{-\frac{\beta}{2\beta + d}}\,.
	\end{equation}
\end{lemma} 
\begin{proof}
	Define $g\J(x) = q\J( h(x + \vec{j}) )$ to be a polynomial function $g:[0,1]^d \to \mb{R}$, noting that $h(x + \vec{j}) \in I\J$. Write 
	\[g\J(x) = \sum_{\alpha: \abs{\alpha} \leq \ell} s_\alpha x^{\alpha}
	\]
	where $s_\alpha \in \mb{R}$ and $\alpha \in \mb{Z}^d_{\geq 0}$ is a multi-index. Let $V_1, \ldots, V_M$ denote the elements of $\mc{P}_\ell$, and let $s$ denote the vector of coefficients of $g\J$, indexed by the multi-index $\alpha$. Define $y \in \mb{R}^M$ by $y_i = g\J(V_i)$. Consider the Vandermonde-type $M\times M$ matrix $\mc{V}$ defined by $\mc{V}_{i, \alpha} := V_i^\alpha$, where the columns of $\mc{V}$ are indexed by the multi-index $\alpha$ with $|\alpha|\leq \ell$. Then it holds that $\mc{V} s = y$. 
	
	By Theorem \ref{thm:classical_interpolate}, the matrix $\mc{V}$ is invertible. Let $\sigma_0>0$ denote its minimum singular value, noting that $\sigma_0$ is a constant that only depends on $\beta$ and $d$. 	
Let $\hat{y} \in \mb{R}^M$ denote the vector defined by $\hat{y}_k = \hat f( U_k^{\vec{j}} )$, where $k \in [M]$ indexes the points in the principal lattice $\mc{P}_\ell$, and let $\hat{s} = \mc{V}^{-1} \hat{y}$. By \eqref{eqn:noisy_interpolate}, 
	\[
		\abs{\hat{y} - y}_2 \leq \sqrt{M} \breve{c} (\log n)^{\frac{1}{2}} n^{-\frac{\beta}{2\beta + d}}. 
	\] 
	
	Hence, 
	\begin{align*}
		\abs{\hat{s} - s}_2^2 &= \abs{\mc{V}^{-1}(\hat{y} - y)}_2^2  \leq \frac{M \breve{c}^2}{\sigma_0} (\log n)^{\frac{1}{2}} n^{-\frac{\beta}{2\beta + d}}
	\end{align*}
	
	Define $\hat{q}\J(x) = \hat{g}\J( x/h - \vec{j} ),$
	and note that $\hat{q}\J$ passes through the points $\{ (U_k\V{j}, \hat f( U_k\V{j} )) \}_{k = 1}^M$. Then for all $x \in I_{\vec{j}}$, by the previous display and Cauchy--Schwarz, we have 
\begin{align*}
\abs{\hat{q}\J(x) - q\J(x)} &\leq \sum_{\alpha: |\alpha| \leq \ell}\abs{(\hat{s}_\alpha - s_\alpha) \left( x/h - \vec{j} \right)^\alpha } \\
&\leq \frac{M \breve{c}}{\sqrt{\sigma_0}} (\log n)^{\frac{1}{2}} n^{-\frac{\beta}{2\beta + d}}\,.
\end{align*}
\end{proof}
	
	Define $\tilde{f}(x) = \sum\J \hat{q}\J(x) \1(x \in I_{\vec{j}})$, and observe that  Theorem~\ref{thm:main} follows from \eqref{eqn:Taylor_approx}, Lemma \ref{lem:interpolation_error}, and the triangle inequality. Though we have derived a high probability bound, the expectation claimed in Theorem \ref{thm:main} follows using the uniform boundedness of H\"{o}lder functions. Tracing constants above and applying Lemma \ref{lem:Vandermonde} below yields the expression for $\tilde{c}$.

\subsection{Stability of the Vandermonde-type matrix}

\label{appendix:singular_values}

We give below an explicit lower bound on the smallest singular value of the Vandermonde-type matrix $\mc{V}$ from Lemma \ref{lem:interpolation_error} associated to the $\ell$-th principal lattice on the simplex. This gives an explicit dependence of our results on the smoothness $\beta$ and dimension $d$. 
	
	\begin{lemma}
	\label{lem:Vandermonde}
		Let $\mc{V}$ denote the Vandermonde-type matrix associated to the $\ell$-th principle lattice $\mc{P}_\ell$ of $\Delta_d$. To be explicit, let $x_1, \ldots, x_M \in [0,1]^d$ denote the elements of $\mc{P}_\ell$, where $M = \binom{\ell + d}{\ell}$. Given a multi-index $\alpha$ with $|\alpha| \leq \ell$, define $\mc{V}_{k, \alpha} := x_k^\alpha$. Let $\sigma_0$ denote the minimum singular value of $\mc{V}$. Then
		\[
		\sigma_0 \geq \binom{\ell + d}{\ell}^{-3} 4^{-\ell} \ell^{-2\ell}. 
		\]
	\end{lemma}

	\begin{proof}
		Write
		\[
		p_k(x) = \sum_{\alpha: \abs{\alpha} \leq \ell} u_{k, \alpha} x^{\alpha}, 
		\]
		where $p_k$ was defined in Theorem  \ref{thm:classical_interpolate} for $1 \leq k \leq M$. 
		
		Then write $u_{k, \alpha} = \omega_k \cdot v_{k, \alpha}$, where
		\begin{equation}
		\label{eqn:normalization}
		\omega_k := \prod_{\substack{t = 0 \\ \lambda_t(x_k) > 0}}^d \prod_{r = 0}^{\ell \lambda_t(x_k) - 1} \frac{1}{\left(\lambda_t(x_k) - \frac{r}{\ell}  \right)}\,. 
		\end{equation}
		Since $\Lambda(x_k) = \frac{1}{\ell}(s_0(x_k), \ldots, s_d(x_k))$, where $s_j(x_k) \in \mb{Z}_{\geq 0}$ for all $0 \leq j \leq d$, we conclude that the denominator of each term of \eqref{eqn:normalization} is at least $1/\ell$. Thus
		\begin{equation}
		\label{eqn:normalization_bound}
		\omega_k \leq \ell^{\ell}, 
		\end{equation} 
		since there are $\ell$ terms in the product defining $\omega_k$.

		Note that
		\[
		\lambda_t(x_k) = \begin{cases}
		x_{kt} \quad &\text{if} \, \, 1 \leq t \leq d \\
		1 - \sum_{t = 1}^d x_{kt} \quad &\text{if} \, \, t = 0. 
		\end{cases}
		\]
		Therefore, for all $1 \leq k \leq M$, we may write
		\[
		\tilde{p}_k(x) := \frac{p_k(x)}{\omega_k} = \prod_{s = 1}^\ell ( m_{s, k} \cdot x + h_{s, k} ), 
		\]
		where $m_{s, k} \in \mb{R}^d$ with $\abs{ m_{s, k} }_{\infty} \leq 1$, and $|h_{s, k}| \leq 1$. It follows that the coefficient $v_{k, \alpha}$ of the monomial $x^\alpha$ in $\tilde{p}_k(x) $ satisfies
		\begin{equation}
		\label{eqn:coefficient_bound}
			| v_{k, \alpha} | \leq \binom{\ell}{|\alpha|}. 
		\end{equation}
		
		Next, observe that $\mc{V}^{-1}_{\alpha, k} = u_{k, \alpha}$. Hence, by \eqref{eqn:normalization_bound} and \eqref{eqn:coefficient_bound}, for all multi-indices $\alpha$ with $|\alpha| \leq \ell$,  
		\begin{equation}
		\label{eqn:invert_bound}
			\sum_{k = 1}^M |u_{k, \alpha}| \leq \binom{\ell + d}{\ell} \binom{\ell}{|\alpha|} \ell^\ell \leq \binom{\ell + d}{\ell} 2^\ell \ell^\ell. 
		\end{equation}
		
		Letting $\eta \in \mb{R}^M$ denote a unit vector and applying the previous inequality,
		\begin{align*}
		\eta^T \mc{V}^{-T} \mc{V}^{-1} \eta &\leq M \cdot  \max_{\alpha} \left( \sum_{k = 1}^M |u_{k, \alpha}| \right)^2 \\
		&\leq \binom{\ell + d}{\ell}^3 4^\ell \ell^{2\ell}.
		\end{align*}
		The conclusion of Lemma \ref{lem:Vandermonde} follows.  
	\end{proof}

    \section{\uppercase{A result of Kolmogorov and Tikhomirov}}
        
    	Given a function class $\cF$, let $N( \mc{F}, \delta)$ denote the minimal number of cubes of side-length $\delta$ that cover $\cF$, and define $H(\mc{F}, \delta) = \log N( \mc{F}, \delta)$ to be the metric entropy. A classical result of \citet{Tik93} shows that 
            \begin{equation}
            \label{eqn:covering_number}
            H( \hd(\beta), \delta) \leq c_{\beta, d, L} \, \delta^{-\frac{d}{\beta}}. 
            \end{equation}
            
            Their proof strategy is conceptually similar to our piecewise multivariate polynomial approximation scheme in that they subdivide the cube as we do here, approximate $f$ by its Taylor polynomial in each cube, and then discretize the coefficients. Our techniques imply a slightly weaker version of the bound \eqref{eqn:covering_number}. 
            
            
            Define a mesh as in steps 1 and 2 of our formal construction of $\ti f$ as in Section \ref{sec:principal_lattice}, but now for a general parameter $h > 0$ to be set later. This mesh has $M h^{-d}$ points that we denote by $\{ U\V{j}_k \}_{\vec{j}, k }$. Let $f, g \in \hd(\beta)$ be such that for all $\vec{j}, k$ it holds that
            \[
            \abs{f(U\V{j}_k) - g(U\V{j}_k)} \leq h^\beta.  
            \]
            By the H\"{o}lder condition and Lemma \ref{lem:taylor_holder}, there exists a degree $\ell = \flo{\beta}$ polynomial $q_{\vec{j}}$ approximating $f$ in $I_{\vec{j}}$ and a degree $\ell = \flo{\beta}$ polynomial $r_{\vec{j}}$ approximating $g$ in $I_{\vec{j}}$, each with error $h^\beta$ pointwise. We conclude that
            \[
            \abs{ q_{\vec{j}}(U\V{j}_k) - r_{\vec{j}}(U\V{j}_k)} \leq c_{\beta, d, L} \, h^{\beta}  
            \]
            for all $\vec{j}, k$. Following the proof of Lemma \ref{lem:interpolation_error} and using the bounds on the singular values of $\cV$, this implies that for all $x \in I_{\vec{j}}$,
            \[
            \abs{ q_{\vec{j}}(x) - r_{\vec{j}}(x) } \leq c_{\beta, d, L} \, h^\beta. 
            \]
            Hence we conclude that for all $x \in [0,1]^d$, 
            \[\abs{f(x) - g(x)} \leq c_{\beta, d, L} \, h^\beta.\]
    
            The H\"{o}lder densities are uniformly bounded by some constant $c_{\beta, d, L}$ (see Lemma \ref{lem:bounded}). Hence setting $\delta = h^{\beta}$ and rounding the values of each density at each point $U\V{j}_k$ to multiples of $\delta$, we see that there exists an $\delta$-net of size at most
            \[
            \left(  \frac{c_{\beta, d, L}}{ \delta } \right)^{M \delta^{-d/\beta}}.
            \]   
            Therefore
            \[
             H( \hd(\beta), \delta) \leq c_{\beta, d, L} \, \delta^{-\frac{d}{\beta}} \log \frac{1}{\delta},
            \]
            a mildly weaker bound than \eqref{eqn:covering_number}. 
        
        \appendix 
        
        \section{\uppercase{Properties of H\"{o}lder densities}}
         \label{appendix:Holder}
                
                For completeness, we provide proofs of standard facts about the class of H\"{o}lder functions $\hd(\beta)$.
                
                
                \begin{lemma}[Inclusion] \label{lem:inclusion} Let $\hd(\beta, d, L)$ denote the class of H\"{o}lder densities in dimension $d$. If $\beta > 1$, then it holds that $\hd(\cbeta, d, L) \subset \hd(\cbeta - 1, d, d^{3/2} L)$.
                \end{lemma}
                
                \begin{proof}
                Let $f \in \hd(\beta, d, L)$. Since $f$ is supported on $[0,1]^d$ and smooth on $\mb{R}^d$, we have that
                \begin{equation}
                \label{eqn:Holder_bd}
                    \abs{ D^s f(x) } \leq  L \abs{x}_2 \leq L\sqrt{d}
                \end{equation}
                for all $|s| = \cbeta$. 
                
                Fix $x, y \in [0,1]^d$, and define for $1 \leq i \leq d+1$ the point $z^i \in [0,1]^d$ to be
                \begin{equation*}
                    z^i_{j} = \begin{cases}
                    x_j \quad \text{if} \,\, j \geq i \\
                    y_j \quad \text{if} \,\, j < i.\\
                    \end{cases}
                \end{equation*}
                Observe that $z^1=x$ and $z^{d+1} = y$. 
                
                Let $t$ denote a multi-index with $|t| = \cbeta - 1$.
                By the fundamental theorem of calculus and the H\"{o}lder condition,
                \begin{equation*}
                |D^t f(x) - D^t f(y)| \leq \sum_{i = 1}^d \abs{ D^t f(z^i) - D^t f(z^{i + 1}) } \\ 
                =\sum_{i = 1}^d \abs{ \int_{x_i}^{y_i} \frac{\partial}{\partial x_i} D^t f(x_1, \ldots, z, y_{i+1}, ..., y_d) \, \mathrm{d}z }.
                \end{equation*}
                Using \eqref{eqn:Holder_bd}, the expression in the second line is bounded above by $Ld^{3/2}$, which proves the lemma.
                
                \end{proof}

                \begin{lemma}[Uniform boundedness]
                \label{lem:bounded}
                The class $\hd(\beta)$ is uniformly bounded. In particular,
                \[
                \sup_{f \in \hd(\beta)} \norm{f}_\infty \leq d^{3\cbeta/2  + 1/2}\, L.
                \]
                \end{lemma}
                
                \begin{proof}
                Suppose first that $f \in \cH(\beta)$ for $\beta > 1$. By repeated application of Lemma \ref{lem:inclusion}, $f$ is $(d^{3\cbeta/2 } L)$-Lipschitz. Since $f$ is supported on $[0,1]^d$,
                \[
                |f(x)| = |f(x) - f(0)| \leq d^{3\cbeta/2 } L \abs{x}_2 \leq d^{3\cbeta/2  + 1/2} L. 
                \]
                If $\beta < 1$, then arguing as in the previous display, we see that $|f(x)| \leq L \sqrt{d}$ for all $x \in \mb{R}^d$. 
                \end{proof}
        
                \begin{lemma}[Taylor approximation]
                        \label{lem:taylor_holder}
                        Given $f \in \hd(\beta)$, let $f_{x, \cbeta}$ denote its Taylor polynomial of degree $\cbeta$ at a point $x \in \mb{R}^d$,
                        $$
                        f_{x, \cbeta}(y)=\sum_{|s| \le \cbeta}\frac{(y -x)^s}{s!}D^sf(x)\,, \quad y \in  \mb{R}^d\,.
                        $$
                        
                        Then it holds that $$
                        \big|f(y)- f_{x, \cbeta}(y)\big| \le \frac{L d^{\cbeta/2} }{\cbeta !} \,  |x-y|_2^{\beta}\,,      \quad x,y \in  \mb{R}^d\,.
                        $$
                        \end{lemma}
                        
                        \begin{proof}
                        By Taylor's theorem with remainder~\citep[see, eg.,][]{Fol99} 
                        \begin{equation*}
                        \big|f(y)- f_{x, \cbeta}(y)\big| = \\ \left| \sum_{|s| = \cbeta} \frac{1}{s!} \left[ D^s f(x + c(y - x)) - D^s f(x) \right] (y - x)^s  \right|
                        \end{equation*}
                        for some constant $c \in (0, 1)$. By the triangle inequality and the H\"{o}lder condition, the expression in the second line is bounded above by 
                        \begin{equation*}
                        \sum_{|s| = \flo{\beta}} \frac{ L \abs{x - y}_2^{\beta - \cbeta} }{s!} \abs{ (y - x)^s } 
                        = \\ \frac{L\abs{x - y}_2^{\beta - \cbeta}}{\cbeta !} \left( \sum_{i = 1}^d |x_i - y_i| \right)^{\cbeta},
                        \end{equation*}
                        where the equality is by the multinomial theorem. In turn, this last expression is bounded above by
                        \[
                        \frac{L d^{\cbeta/2} }{\cbeta !} \abs{x - y}_2^{\beta}
                        \]
                        using Cauchy--Schwarz.
                        \end{proof}

\bibliographystyle{plainnat}
\bibliography{Interpolation_arxiv_bib}	
	
\end{document}